\documentclass[12pt]{amsart}
\usepackage{amsmath,amsthm,amssymb,amsfonts,enumerate,color}
\usepackage{mathrsfs}
\usepackage{amstext,amsxtra}
\usepackage{txfonts} 
\usepackage[colorlinks,linkcolor=blue,citecolor=blue,hypertexnames=false]{hyperref}
\allowdisplaybreaks
\usepackage{pgf,tikz}
\usepackage{stmaryrd}
\newtheorem{thm}{Theorem}[section]
\newtheorem{prop}[thm]{Proposition}
\newtheorem{cor}[thm]{Corollary}

\newtheorem{lem}[thm]{Lemma}
{\theoremstyle{definition}
\newtheorem{exa}[thm]{Example}
\newtheorem{rem}[thm]{Remark}}

\newcommand{\pic}{\operatorname{Pic}}

\newcommand{\cl}{\mathcal{C}}

\newcommand{\zl}{\mathcal{Z}}

\newcommand{\pl}{\mathcal{P}}

\newcommand{\yl}{\mathcal{Y}}

\newcommand{\rb}{\mathbb{R}}

\newcommand{~}{\quad}
\newcommand{\cb}{\mathbb{C}}

\newcommand{\zb}{\mathbb{Z}}

\newcommand{\glt}{\geqslant}
\newcommand{\llt}{\leqslant}
\newcommand{\undl}{\underline}

\begin{document} 

\title[Welschinger invariants under real surgeries]
{Higher genus Welschinger invariants under real surgeries}

\author{Yanqiao Ding}

\address{School of Mathematics and Statistics\\ Zhengzhou University\\
                       Zhengzhou, 450001 \\ P. R. China}
\email{yqding@zzu.edu.cn}



\begin{abstract}
  According to \cite{brugalle2016}, a real surgery of a real del Pezzo surface $X_\rb$
  along a real sphere $S$ is a modification of the real structure on $X_\rb$ in a neighborhood of $S$.
  In this paper, we study the behavior of higher genus Welschinger invariants under real surgeries,
  and obtain a genus decreasing formula of Welschinger invariants.
\end{abstract}

\maketitle


\section{Introduction}
Let $(X,\tau_X)$ be a real algebraic surface.
Denote by $H^{\tau_X}_2(X;\zb)$ the space of $\tau_X$-invariant classes in $H_2(X;\zb)$,
and by $H^{-\tau_X}_2(X;\zb)$ the space of $\tau_X$-anti-invariant classes in $H_2(X;\zb)$.
$S\subset X$ is called a \textit{real sphere} of $X$ if $S$ is a sphere globally invariant under $\tau_X$.
Let $\Delta\subset\cb$ be a small disk endowed with complex conjugation,
and let $\pi:\zl\to \Delta$ be a flat real morphism from a smooth real algebraic
manifold of complex dimension $3$. Suppose the fiber $\zl_\tau$, $\tau\in\Delta^*=\Delta\setminus\{0\}$,
is a non-singular projective algebraic surface, and the central fiber $\zl_0$ is
a real algebraic surface with one singular point $z$ of type $A_1$ (node).
Suppose in addition that $\zl$ is locally given at $z$ by the equation
$$
x_1^2+x_2^2\pm x_3^2=\tau,~(x_1,x_2,x_3,\tau)\in\cb^4,
$$
and that $\pi$ is locally given by $\pi(x_1,x_2,x_3,\tau)=\tau$.
There exists a real sphere $S\subset\zl_\tau$ such that
$\zl_\tau$ and $\zl_{-\tau}$ are obtained one from the other by a modification of the real structure
in a neighborhood of $S$ when $\tau\in\rb\Delta^*$.
The real algebraic surface $Y_\rb=\zl_\tau$ is called the \textit{real surgery of $X_\rb=\zl_{-\tau}$ along $S$}.
We refer the readers to Section $\ref{sec:nodal deg}$ or \cite[Section $4.2$]{iks2013a} for more details,
and to \cite{brugalle2016} for the symplectic case: real surgeries along real Lagrangian spheres.
The class $[S]$ in $H_2(X;\zb)$ is $\tau_X$-anti-invariant if and only if it is
$\tau_Y$-invariant (in this case $\chi(\rb Y_\rb)=\chi(\rb X_\rb)+2$)
(see \cite[Section 1.3]{brugalle2016}).

The behavior of Welschinger invariants under real surgeries has
many important applications such as: the computations, the vanishing results,
the positivity and asymptotic properties.
Itenberg-Kharlamov-Shustin \cite{iks2013a} applied the degeneration technique
to study the positivity and asymptotics of Welschinger invariants of
real del Pezzo surfaces of degree $\geq 2$.
By applying a real version of symplectic sum formula \cite{egh2000,ip2004,lr2001,li2002},
Brugall\'{e}-Puignau \cite{bp2013,bp2014} investigated the relations of genus $0$
Welschinger invariants of two real symplectic manifolds $X_\rb$ and $Y_\rb$,
where $Y_\rb$ is the real surgery of $X_\rb$ along a real Lagrangian sphere $S$.
Through combining the relations from \cite{bp2014} with the technique of floor diagrams relative to a conic,
E. Brugall\'e \cite{bru2015} obtained some explicit computations of  Gromov-Witten invariants
and Welschinger invariants of some del Pezzo surfaces.
By excluding the ``bad'' bifurcation, E. Shustin \cite{shustin2015} extended the definition of
Welschinger invariants to higher genus case for real algebraic del Pezzo surfaces.
Through adapting the proof of \cite{shustin2015} in the strategy proposed in \cite{wel2005a},
E. Brugall\'e \cite{brugalle2016} proposed a way to define the higher genus
Welschinger invariants in the symplectic category.
Note that except the class $d$ and the number of conjugated pairs,
genus $0$ Welschinger invariants also depend on the choice of a connected
component $L$ of $\rb X_\rb$, and of a class $F\in H_2^{\tau_X}(X\setminus L;\zb/2\zb)$
(see \cite{brugalle2016,iks2014,shustin2015,wel2005a,wel2015}).
When the chosen connected component $L$ is disjoint from
the Lagrangian sphere $S$, and the class $F$ is normal to $[S]$,
then the relations from \cite{bp2014} only involve genus $0$
Welschinger invariants of $X_\rb$ and $Y_\rb$.
Observed this fact, E. Brugall\'{e} \cite{brugalle2016} obtained
a surprisingly very simple relation among genus $g$ Welschinger
invariants of the real symplectic 4-manifolds differing by a special kind of real surgery.

The main purpose of the present paper is to provide another important
observation about the higher genus Welschinger invariants of two real del Pezzo surfaces
$X_\rb$ and $Y_\rb$ which are related by a real surgery.
When the real sphere $S\subset\rb Y_\rb$,
which is disjoint with a chosen subset $\undl L'$ of $g$ connected components of $\rb X_\rb$,
contains one real point of the real configuration $\undl x$ and $F$ is normal to $[S]$,
there is a surprising phenomenon of decreasing genera. Namely,
the Welschinger invariants $W_{Y_\rb,(\undl L',S),F}$ with genus $g$ of $Y_\rb$ is equal to a certain
combination of Welschinger invariants $W_{X_\rb,\undl L',F}$ of $X_\rb$ with smaller genus.
Note that the genus is indexed by the number of the chosen connected real components.
We refer the readers to Section $\ref{subsec:hgwel-alg}$ or \cite{shustin2015}
for the definition of higher genus Welschinger invariants of real del Pezzo surfaces
(See also \cite{brugalle2016} for the symplectic case).

\begin{thm}\label{thm:gdf-alg}
Let $X_\rb$, $Y_\rb$ be two real del Pezzo surfaces of the same degree $\glt2$
with $\rb X\neq\emptyset$, $\rb Y=\rb X\sqcup S^2$,
and $S$ be a real sphere in $X_\rb$ realizing a class in $H^{-\tau_X}_2(X;\zb)$.
Let $\undl L=(L_0,\ldots,L_g)$, $\undl L'=(L_0,\ldots,L_{g-1})$,
where $L_0,\ldots,L_{g-1}$ are $g$ distinct connected components of $\rb X_\rb$ which is disjoint from $S$,
$L_g=S^2\subset\rb Y_\rb$, and $L=\cup_{i=0}^{g-1}L_i$.
Assume $F\in H^{\tau_Y}_2(X\setminus(L\cup S);\zb/2\zb)$ is normal to $[S]$.
Suppose the class $d\in H^{-\tau_Y}_2(X;\zb)$ satisfies $c_1(X)\cdot d+g-2>0$ and $d\neq l[E_i]$,
where $[E_i]$ are the exceptional divisors. Let $\undl r=(r_0,\ldots,r_{g-1},1)$,
$\undl r'=(r_0,\ldots,r_{g-1})$. Then
\begin{equation}\label{eq:gdf}
W_{Y_\rb,\undl L,F}(d,\undl r)=\sum_{k\glt1}(-1)^{(k-1)}k^2W_{X_\rb,\undl L',F}(d-k[S],\undl r'),
\end{equation}
whenever the invariant $W_{Y_\rb,\undl L,F}(d,\undl r)$ is defined.
\end{thm}

Because of \cite[Lemma 2.7]{brugalle2016}, all invariants $W_{X_\rb,\undl L',F}(d-k[S],\undl r')$
are also defined if $F$ is normal to $[S]$ in $H_2(X;\zb/2\zb)$ and $W_{Y_\rb,\undl L,F}(d,\undl r)$
is defined.

\begin{rem}
By using the definition of higher genus Welschinger invariants of symplectic $4$-manifolds
proposed by E. Brugall\'e \cite{brugalle2016},
one can prove formula $(\ref{eq:gdf})$ is also true in the symplectic case.
See the Appendix for an explicit expression of Theorem $\ref{thm:gdf-alg}$ in the symplectic case.
Some blow-up formulas of higher genus Welschinger invariants can also be obtained
which generalize the genus $0$ results \cite{dh2016a}.
Formally, \cite[Theorem 1.1,Theorem 1.2,Theorem 1.4]{dh2016a} are stated for $g=0$
and connected $\rb X$. However, neither $g$ nor the number of connected components of $\rb X$
plays a role in the proof if we replace \cite[Proposition 2.6]{dh2016a} by Lemma $\ref{lem:mspace-symp}$.
The readers may refer to \cite{dh2016a} for more details.
\end{rem}

\medskip\noindent
{\bf Acknowledgment}: The author is particularly grateful to Prof. Jianxun Hu
for his continuous support and encouragement as well as enlightening discussions,
and to Erwan Brugall\'e for pointing out a mistake of the first manuscript of this paper
and sharing the situation of the definition of higher genus Welschinger invariants.
The research is partially supported by Startup Research Fund of Zhengzhou University (No. $32210405$).

\section{Real surgeries and enumeration of curves}

\subsection{Nodal degenerations and real surgeries}\label{sec:nodal deg}
In this subsection, we review the nodal degeneration of del Pezzo surfaces
based on \cite[Section $4.2$]{iks2013a}.

Let $\Delta\subset\cb$ be a small disk,
and let $\pi:\zl\to \Delta$ be a holomorphic map from a smooth $3$-dimensional variety.
Suppose the fiber $\zl_\tau$, $\tau\in\Delta^*=\Delta\setminus\{0\}$, is a del Pezzo surface,
and the central fiber $\zl_0$ is a surface with one singular point $z$ of type $A_1$ (node).
With respect to appropriate local coordinates $x_1,x_2,x_3$ at the point $z$,
the map $\pi$ is given by
$$
\pi(x_1,x_2,x_3)=a_1x_1^2+a_2x_2^2+a_3x_3^2,~a_1a_2a_3\neq0.
$$
$\pi$ is a submersion at each point of $\zl\setminus\{z\}$.
The family $\pi:\zl\to\Delta$ is called \textit{nodal degeneration}.

Make the base change $\tau=t^2$ (resp. $\tau=-t^2$),
and let $\zl'=\Delta\times_{t^2=\pi}\zl$ (resp. $\yl'=\Delta\times_{-t^2=\pi}\zl$).
Perform the blow up $\widetilde\zl\to\zl'$ (resp. $\widetilde\yl\to\yl'$) at the node of the new family,
and obtain a family $\tilde\pi:\widetilde\zl\to\Delta$ (resp. $\tilde\pi':\widetilde\yl\to\Delta$),
whose fibers $\widetilde\zl_t$ (resp. $\widetilde\yl_t$), $t\neq0$, are del Pezzo surfaces,
and central fiber $\widetilde\zl_0=Z\cup (\cb P^1\times\cb P^1)$
(resp. $\widetilde\yl_0=Z\cup (\cb P^1\times\cb P^1)$),
$E=Z\cap (\cb P^1\times\cb P^1)$ is a smooth rational $(-2)$-curve in $Z$,
and $(Z,E)$ is a nodal del Pezzo pair (i.e. the anti-canonical class $-K_Z$ is effective,
positive on all curves different from $E$, and $K_ZE=0$).
In $H_2(\cb P^1\times\cb P^1;\zb)$, the curve $E$ represents the class $l_1+l_2$,
where $l_1=[\cb P^1\times\{p\}]$ and $l_2=[\{p\}\times\cb P^1]$
are the generators of $H_2(\cb P^1\times\cb P^1;\zb)$.
Note that $l_1$ and $l_2$ are well defined in $H_2(\cb P^1\times\cb P^1;\zb)$ up to interchanging with each other.
The family $\tilde\pi:\widetilde\zl\to\Delta$ (resp. $\tilde\pi':\widetilde\yl\to\Delta$)
is called the \textit{unscrew} (resp. \textit{mirror unscrew}) of the nodal degeneration
$\pi:\zl\to\Delta$.

If the nodal degeneration $\pi:\zl\to\Delta$ has a real structure
$\tau_\zl$ which lifts the standard complex conjugation, the point $z$ is
real and the map $\pi$ can be represented by
\begin{equation}\label{eq:real-quadratic}
\pi(x_1,x_2,x_3)=a_1x_1^2+a_2x_2^2+a_3x_3^2,~a_1,a_2,a_3\in\rb,~a_1a_2a_3\neq0,
\end{equation}
with respect to appropriate real local coordinates at $z$.
The real structure $\tau_\zl$ induces two real structures on the unscrew
$\tilde\pi:\widetilde\zl\to\Delta$ (resp. mirror unscrew $\tilde\pi':\widetilde\yl\to\Delta$).
The real unscrew (resp. mirror unscrew) with the real structure which covers the complex conjugation
$t\mapsto\bar t$ is called a \textit{$\theta$-unscrew} (resp. \textit{$\theta$-mirror unscrew}),
where $\theta$ is the signature of the quadratic form $(\ref{eq:real-quadratic})$.
The quadric $\cb P^1\times\cb P^1\subset\widetilde\zl_0$ is real, and
\begin{equation}
\rb (\cb P^1\times\cb P^1)\simeq
\left\{\begin{aligned}
&S^2,~~\text{ if }\theta=3\text{ or }-1,\\
&(S^1)^2,~\text{ if }\theta=1,\\
&\emptyset,~~~\text{if }\theta=-3.
        \end{aligned} \right.
\end{equation}
The other real structure on the unscrew (resp. mirror unscrew) covers the conjugation $t\mapsto-\bar t$.


Let $\tilde\pi:\widetilde\zl\to\Delta$ (resp. $\tilde\pi':\widetilde\yl\to\Delta$)
be the $\theta$-unscrew (resp. $(-\theta)$-mirror unscrew) of a nodal degeneration $\pi:\zl\to\Delta$.
Let $(Y,\tau_Y)$ (resp. $(X,\tau_X)$) be a real del Pezzo surface which is real deformation equivalent
to a real fiber $\widetilde\zl_t$ (resp. $\widetilde\yl_t$), $t\in\rb\Delta^*$,
of the $\theta$-unscrew (resp. $(-\theta)$-mirror unscrew).
Following the notation of real surgery along a real Lagrangian sphere introduced by
E, Brugall\'e in \cite{brugalle2016}, the real surface $(Y,\tau_Y)$ is called a
\textit{real surgery of $(X,\tau_X)$ along a real sphere $S$}.
The notation $Y_\rb \xrightarrow[]{S}  X_\rb$ means that
$X_\rb$ and $Y_\rb$ are related by a real surgery along the real sphere $S$
and $\chi(\rb Y_\rb)=\chi(\rb X_\rb)+2$.

\begin{exa}
If $X_\rb$ and $Y_\rb$ are two real del Pezzo surfaces of degree $2$
with $\rb X\neq\emptyset$ and $\rb Y=\rb X\sqcup S^2$,
$Y_\rb$ is the real surgery of $X_\rb$ along a real sphere.
Denote by $Q_X$, $Q_Y$ the real nonsingular quartic curves corresponding to
the real del Pezzo surfaces $X_\rb$, $Y_\rb$ of degree $2$ respectively (see \cite[Section $2.2$]{iks2013a} for details).
Note that the real part of a real nonsingular quartic curve is isotopic in $\rb P^2$
either to the union of $0\llt a\llt4$ null-homologous circles
placed outside each other or to a pair of null-homologous circles placed one inside the other.
In fact, the del Pezzo surfaces $X_\rb$, $Y_\rb$ of degree $2$ can be included into the following
(mirror) unscrews of real nodal degenerations corresponding to nodal degenerations of quartics.
\begin{itemize}
  \item Case $X_\rb$: we degenerate the quartic $Q_X$ into a nodal quartic $Q_X'$ with real part
  $\rb Q_X'=\rb Q_X\sqcup\{pt\}$, and choose a $(-3)$-mirror unscrew.
  \item Case $Y_\rb$: we degenerate the quartic $Q_Y$ into a nodal quartic so that one of the ovals
  collapses to a point, and choose a $3$-unscrew.
\end{itemize}
\end{exa}


\subsection{Higher genus Welschinger invariants}\label{subsec:hgwel-alg}
Let $X$ be a del Pezzo surface with a real structure $\tau_X$
such that $\rb X$ contains at least $g+1$ connected components
$L_0,\ldots,L_g$ for some $g\glt1$. Let $L=L_0\cup\ldots\cup L_g$,
$\undl L=(L_0,\ldots,L_g)$.
Denote by $\pic^\rb(X)\subset\pic(X)$ the subgroup of real divisor classes.
For each $D\in\pic^\rb(X)$, there exists a class
$l_{L_i,D}=[\rb C\cap L_i]\in H_1(L_i;\zb/2\zb)$,
where $C\in|D|$ is any real curve (cf. \cite{shustin2015}).

A divisor class $D\in\pic^\rb(X)$ is called \textit{L-compatible},
if $l_{G,D}=0$ for any connected component $G\subset\rb X\setminus L$.
Given $F\in H _2^{\tau_X}(X\setminus L;\zb/2\zb)$,
and choose a class $d\in H_2(X;\zb)$ corresponding to a big and nef, $L$-compatible divisor class
$D\in\pic^\rb(X)$ such that
\begin{equation}\label{eq:alg-class}
\frac{d^2-c_1(X)\cdot d}{2}+1\glt g \text{~and~} c_1(X)\cdot d\glt g+1-\sum_{i=0}^{g}l_{L_i,d}^2.
\end{equation}
Let $\pl_{\undl r,m}(X,\undl L)$ (where $\undl r=(r_0,\ldots,r_g)$)
be the space of configurations consisting of $r_i$ real points in $L_i$, $i=0,\ldots,g$,
and $m$ pairs of complex conjugated points in $X\setminus\rb X$.
Choose $(\undl r,m)$ such that
\begin{equation}\label{eq:alg-number}
c_1(X)\cdot d+g-1=r_0+\ldots+r_g+2m ~\text{and}~ r_i\equiv l_{L_i,d}^2+1 \mod 2~~ \forall i\in\{0,...,g\}.
\end{equation}


Assume $\undl x\in\pl_{\undl r,m}(X,\undl L)$.
Denote by $\cl(X,d,\undl L,\undl{x})$ the set of real holomorphic curves $f:\Sigma_g \to X$
of genus $g$ in $X$ realizing the class $d$,
passing through $\undl{x}$, and such that $f(\rb\Sigma_g)\subset L$.
Condition $(\ref{eq:alg-number})$ implies that each $L_i$ contains a
connected component of $f(\rb\Sigma_g)$,
and $\Sigma_g$ is a maximal real curve (\textit{i.e.} $\rb\Sigma_g$ has exactly $g+1$ connected components).
If $X$ is sufficiently generic in its deformation class,
and $\undl x\in\pl_{\undl r,m}(X,\undl L)$ is generic,
it follows from \cite[Lemma A$.3$]{shustin2015} that the set
$\cl(X,d,\undl L,\undl{x})$ is finite and composed of real irreducible immersed curves.
Then the integer
$$
W_{X,\undl L,F}(d,\undl r)=\sum_{f\in\cl(X,d,\undl L,\undl{x})}(-1)^{m_{L,F}(f)}
$$
depends neither on $\undl x$, nor on the deformation class of $(X,d,\undl L, F)$
\cite[Theorem $2.1$]{shustin2015}. The $(L,F)$-mass $m_{L,F}(f)$ is defined as
$m_{L,F}(f)=m_L(f)+[f(\Sigma_g^+)]\cdot F$,
where $m_L(f)$ is the number of real isolated nodes
(\textit{i.e.} nodes with two $\tau_X$-conjugated branches) of $f(\Sigma_g)$ in $L$,
and $\Sigma_g^+$ is a half of $\Sigma_g\setminus\rb\Sigma_g$.

\subsection{Relative invariants of monic log-del Pezzo pair}\label{subsec:rel wel-alg}
Let $Z$ be a smooth rational surface which is a blow-up of $\cb P^2$,
and $E\subset Z$ be a smooth rational curve.
Assume $(Z,E)$ is a monic log-del Pezzo pair \textit{i.e.} $-K_Z$ is positive on all curves different from $E$,
$K_ZE\glt0$, $-(K_Z+E)$ is nef and effective, and $(K_Z+E)^2=0$.
The surface $Z$ can be considered as the
plane blown up at $n$ points, $n\glt6$, on a smooth conic and at $1$ point outside the conic,
and $E$ is the strict transform of the conic (see \cite[Section $3.1$]{iks2013a} for details).

An isomorphism between two algebraic maps $f_1:C_1\to Z$ and $f_2:C_2\to Z$
is an isomorphism $\phi:C_1\to C_2$ such that $f_1=f_2\circ\phi$.
Maps are always considered up to isomorphisms.
Given a vector $\alpha=(\alpha_i)_{1\llt i<\infty}\in\zb_{\glt0}^\infty$,
we use the following notations:
$$
|\alpha|=\sum_{i=1}^{+\infty}\alpha_i,  \,\,\,\,\,
I\alpha=\sum_{i=1}^{+\infty}i\alpha_i.
$$
For $k\in \mathbb{Z}_{\glt 0}$ and $\alpha = (\alpha_i)_{1\llt i <\infty}$,
denote $k\alpha := (k\alpha_i)_{1\llt i <\infty}$.
Denote by $\delta_i$ the vector in $\zb_{\glt0}^\infty$ whose all coordinates are
$0$ except the $i$th one which is equal to $1$.

Let $Z$ be the blow up of $\cb P^2$ at $6$ points on a smooth conic
and at $1$ point outside the conic. Suppose that $E$ is the strict transform of the conic.
Denote by $E_i$, $i=1,\ldots,6$, the exceptional divisors of the blow-up at the $6$
points on the conic, and by $E_7$ the exceptional divisor of the blow-up at the point outside the conic.
Let $d\in H_2(Z;\zb)$, $g\in\zb_{\geq0}$,
and $\alpha$, $\beta\in\zb_{\geq0}^\infty$ such that
\begin{equation}\label{eq:tan con-alg}
I\alpha+I\beta=d\cdot[E].
\end{equation}
Choose a configuration $\undl x=\undl x^\circ\sqcup\undl x_{E}$ of points in $Z$,
with $\undl x^\circ$ a configuration of
$c_1(Z)\cdot d-1+g- d\cdot[E]+|\beta|$ points in $Z\setminus E$,
and $\undl x_{E}=\{p_{s,t}\}_{0<t\leq\alpha_s,s\geq1}$ a configuration
of $|\alpha|$ points in $E$.
Denote by $\cl^{(\alpha,\beta)}(Z,E,d,g,\undl x)$ the set
of holomorphic maps $f:\Sigma_g\to Z$ such that

$\bullet$ $\Sigma_g$ is a connected algebraic curve of arithmetic genus $g$;

$\bullet$ $f_*[\Sigma_g]=d$;

$\bullet$ $\undl x\subset f(\Sigma_g)$;

$\bullet$ $E$ is not a component of $f(\Sigma_g)$;

$\bullet$ $f^*(E)=\sum_{s\glt1}\sum_{t=1}^{\alpha_s}sq_{s,t}
+\sum_{s\glt1}\sum_{t=1}^{\beta_s}s\tilde q_{s,t}$, with $f(q_{s,t})=p_{s,t}$.\\
Let
$$
\cl_*^{(\alpha,\beta)}(Z,E,d,g,\undl x)=\{f(\Sigma_g)|
(f:\Sigma_g\to Z)\in\cl^{(\alpha,\beta)}(Z,E,d,g,\undl x)\}.
$$
%
%

Suppose $Z$ has a real structure $\tau_Z$
such that $\rb Z$ contains at least $g+1$ connected components
$L_0,\ldots,L_g$ for some $g\glt1$, and the real rational curve
$E$ has vanish real part $\rb E=\emptyset$. Let $L=L_0\cup\ldots\cup L_g$,
$\undl L=(L_0,\ldots,L_g)$.

Choose a big and nef, $L$-compatible divisor class $D\in\pic^\rb(Z)$ such that
\begin{equation}\label{eq:alg-class-rel}
\frac{d^2-c_1(Z)\cdot d}{2}+1\glt g,~ c_1(Z)\cdot d\glt g+1-\sum_{i=0}^{g}l_{L_i,d}^2,
\text{~and~} d\neq l[E_i], i=1,\ldots,7,
\end{equation}
where $d\in H_2(Z;\zb)$ is the homology class corresponding to the divisor class $D$
under the natural isomorphism.
Choose $(\undl r,m)$ such that
\begin{equation}\label{eq:alg-number-rel}
c_1(Z)\cdot d+g-1=r_0+\ldots+r_g+2m, ~r_i\equiv l_{L_i,d}^2+1 \mod 2~~ \forall i\in\{0,...,g\}.
\end{equation}

Choose $\undl x\in\pl_{\undl r,m}(Z,\undl L)$,
and denote by $\cl(Z,E,d,\undl L,\undl{x})$ the set of real holomorphic curves $f:\Sigma_g \to Z$
of genus $g$ such that $f(\Sigma_g)$ realizes the class $d$ in $Z$, does not contain $E$,
contains $\undl{x}$, and such that $f(\rb\Sigma_g)\subset L$.
Condition $(\ref{eq:alg-number-rel})$ implies that each $L_i$ contains a
connected component of $f(\rb\Sigma_g)$, and $\Sigma_g$ is a maximal real curve.
If $\undl x\in\pl_{\undl r,m}(Z,\undl L)$ is generic,
it follows from \cite[Proposition $2.1$]{ss2013} that the set
$\cl(Z,E,d,\undl L,\undl{x})$ is finite and composed of real irreducible immersed curves.
Then the integer
\begin{equation}\label{eq:rel-inv-alg}
W_{Z,\undl L,F}^E(d,\undl r)=\sum_{C\in\cl(Z,E,d,\undl L,\undl{x})}(-1)^{m_{L,F}(C)}
\end{equation}
depends neither on $\undl x$, nor on the deformation class of $(Z,d,\undl L,E,F)$
\cite[Theorem $3.6$]{brugalle2016}.

\begin{rem}
The integer $W_{Z,\undl L,F}^E(d,\undl r)$ defined by equation $(\ref{eq:rel-inv-alg})$
is an algebraic case of the relative invariant defined by \cite[Theorem $3.6$]{brugalle2016}.
Note that if we replace \cite[Lemma $3.1$]{brugalle2016} by \cite[Proposition $2.1$]{ss2013},
we can prove that the integer $W_{Z,\undl L,F}^E(d,\undl r)$ depends neither on $\undl x$,
nor on the deformation class of $(Z,d,\undl L,E,F)$ similarly to \cite[Theorem $3.6$]{brugalle2016}.
For the convenience of the readers, we recall two theorems
which play a key role in the proof of \cite[Theorem $3.6$]{brugalle2016}
in our situation.
\end{rem}

Suppose that $X_\rb$ and $Y_\rb$ are two real del Pezzo surfaces of degree $2$ and $Y_\rb \xrightarrow[]{S}  X_\rb$.
\begin{thm}{\cite[Theorem $3.7$]{brugalle2016}}\label{thm:corr-alg}
\begin{equation}\label{eq:thmcorr1-alg}
W_{X,\undl L,F}(d,\undl r)=\sum_{k\glt0}\binom{\frac{1}2 d\cdot [E]
  +2k}{k}W^{E}_{Z,\undl L,F}(d-2k[E],\undl r).
\end{equation}
\end{thm}
Using the convention
$$
\binom{-1}{0}=0,
$$
E. Brugall\'e \cite{brugalle2016} proved the following corollary.
\begin{cor}{\cite[Corollary $3.8$]{brugalle2016}}\label{cor:corr-alg}
The number $W^{E}_{Z,\undl L,F}(d,\undl r)$ can be expressed in terms of
the numbers $W_{X,\undl L,F}(d-2k[E],\undl r)$ with $k\glt0$.
More precisely, we have
$$
W^{E}_{Z,\undl L,F}(d,\undl r)=\sum_{k\glt 0}(-1)^{k}\left(
\binom{\frac{1}2 d\cdot [E] +k}{\frac{1}2 d\cdot [E]} +
\binom{\frac{1}2 d\cdot [E] +k-1}{\frac{1}2 d\cdot [E]}
\right) W_{X,\undl L,F}(d-2k[E],\undl r).
$$
\end{cor}

Suppose $E$ is a smooth rational curve in $Z=\cb P^1\times\cb P^1$ representing
the class $l_1+l_2\in H_2(Z;\zb)$.
%
%
%
%
%
We need to investigate the composition of the elements of the set
$\cl^{(\alpha,\beta)}(Z,E,d,g,\undl x)$ with $|\undl x^\circ|\llt1$.
Note that the set $\cl^{(\alpha,\beta)}(Z,E,d,g,\undl x)$ is defined similarly to the above.
The following proposition is a higher genus version of \cite[Proposition $3.4$]{bp2014}.
\begin{prop}{\cite[Section $3$]{v2000}}\label{prop:hgp1p1-alg}
Suppose $Z=\cb P^1\times\cb P^1$, $[E]=l_1+l_2$ and $d\neq l[l_i]$ with $l\glt2$.
then the set $\cl^{(\alpha,\beta)}(Z,E,d,g,\undl x)$ with $|\undl x^\circ|\llt1$
and $d\neq0\in H_2(Z;\zb)$ is empty for a generic configuration $\undl x$,
except in the following situations:
\begin{itemize}
  \item $\cl^{(\delta_1,0)}(Z,E,l_i,0,\undl x_E)$, $\cl^{(0,\delta_1)}(Z,E,l_i,0,\{p\})$, $i=1,2$;
  \item $\cl^{(2\delta_1,0)}(Z,E,l_1+l_2,0,\{p\}\sqcup\undl x_E)$;
  \item $\cl^{(\delta_2,0)}(Z,E,l_1+l_2,0,\{p\}\sqcup\undl x_E)$.
\end{itemize}
Moreover each set of the above four consists of a unique element which is an embedding.
\end{prop}

\section{Genus decreasing formula of Welschinger invariants}\label{subsec:gdf-alg}
Suppose that $X_\rb$ and $Y_\rb$ are two real del Pezzo surfaces of degree $2$ and $Y_\rb \xrightarrow[]{S}  X_\rb$.
From Section $\ref{sec:nodal deg}$, we know $(Y,\tau_Y)$ (resp. $(X,\tau_X)$)
is a real del Pezzo surface of degree $2$ which is real deformation equivalent
to a real fiber $\widetilde\zl_t$ (resp. $\widetilde\yl_t$), $t\in\rb\Delta^*$, of the $3$-unscrew
$\tilde\pi:\widetilde\zl\to\Delta$ (resp. $(-3)$-mirror unscrew $\tilde\pi':\widetilde\yl\to\Delta$) .
In the central fiber $\widetilde\zl_0=Z\cup (\cb P^1\times\cb P^1)$,
$E=Z\cap (\cb P^1\times\cb P^1)$ is a smooth rational $(-2)$-curve in $Z$ such that
$(Z,E)$ is a nodal del Pezzo pair, and $\rb(\cb P^1\times\cb P^1)=S^2$.
$(Z,E)$ is also a monic log-del Pezzo pair.
Let $L=\cup_{i=0}^{g-1}L_i$ be a union of $g$ distinct connected components of $\rb X$,
$L_g=S\subset\rb Y$ be a real sphere disjoint from $L$,
$\undl L=(L_0,\ldots,L_{g-1},L_g)$ and $\undl L'=(L_0,\ldots,L_{g-1})$.

Let $d\in H_2(\widetilde\zl_t;\zb)$ be a homology class corresponding to
a big and nef, $L\cup S$-compatible divisor class $D\in\pic^\rb(\widetilde\zl_t)$.
Choose $\undl x(t)$ a set of $c_1(X)\cdot d+g-1$ real
holomorphic sections $\Delta\to\widetilde\zl$ such that $\undl x(0)\cap E=\emptyset$.
Let $\undl L_\sharp$ be the degeneration of $\undl L$ as $Y$ degenerates to $\widetilde\zl_0$.
Denote by $\cl^\cb(\widetilde\zl_0,d,g,\undl{x}(0))$
the set $\{\bar f:\bar C\to \widetilde\zl_0\}$ of limits, as stable maps,
of maps in $\cl^\cb(\widetilde\zl_t,d,g,\undl{x}(t))$ as $t$ goes to $0$.
Where $\cl^\cb(\widetilde\zl_t,d,g,\undl{x}(t))$
is the set of holomorphic maps $f:C \to \widetilde\zl_t$
with $C$ a connected algebraic curve of arithmetic genus $g$,
such that $f(C)$ realizes the class $d$ and contains $\undl{x}(t)$.
If $\widetilde\zl_t$ and $\undl x(t)$ is generic,
it follows from \cite[Lemma A$.3$]{shustin2015} that the set
$\cl^\cb(\widetilde\zl_t,d,g,\undl{x}(t))$ is finite and composed of irreducible immersed curves.
Let $\cl^\cb_*(\widetilde\zl_0,d,g,\undl{x}(0))=\{\bar f(\bar C)|
(\bar f:\bar C\to \widetilde\zl_0)\in \cl^\cb(\widetilde\zl_0,d,g,\undl{x}(0))\}$.
We know $\bar C$ is a connected nodal curve of arithmetic genus $g$ such that:

$\bullet$ $\undl x(0)\subset\bar f(\bar C)$;

$\bullet$ any point $p\in\bar f^{-1}(E)$ is a node of $\bar C$ which is the intersection of two
irreducible components $\bar C'$ and $\bar C''$ of $\bar C$, with $\bar f(\bar C')\subset Z$
and $\bar f(\bar C'')\subset\cb P^1\times\cb P^1$;

$\bullet$ if in addition neither $\bar f(\bar C')$ nor $\bar f(\bar C'')$ is entirely mapped into $E$,
then the multiplicities of intersection of both $\bar f(\bar C')$ and $\bar f(\bar C'')$ with $E$ are equal.

Given an element
$\bar f:\bar C\to \widetilde\zl_0$ of $\cl^\cb(\widetilde\zl_0,d,g,\undl x(0))$,
denote by $C_1$ (resp. $C_0$)
the union of the irreducible components of $\bar C$ mapped into $Z$ (resp. $\cb P^1\times\cb P^1$).

\begin{lem}\cite[Lemma $3.12$]{brugalle2016}\label{lem:abv-alg}
Given an element $\bar f:\bar C\to \widetilde\zl_0$ of $\cl^\cb(\widetilde\zl_0,d,g,\undl x(0))$,
there exists $k\in\zb_{\glt0}$ such that
$$
\bar f_*[C_1]=d-k[E]~\textrm{ and }~\bar f_*[C_0]=kl_1+(d\cdot[E]+k)l_2.
$$
Moreover $c_1(Z)\cdot\bar f_*[C_1]=c_1(\zl_\lambda)\cdot d$.
\end{lem}

\begin{lem}\cite[Corollary $5.4$]{bru2015}\label{lem:deg-alg}
Suppose that $d\neq l[E_i]$ with $l\glt2$, where $E_i$ is the exceptional divisor in $Z$,
then any irreducible component
of $\bar C$ entirely mapped to $E_i$ is isomorphically mapped to $E_i$.
\end{lem}

\begin{rem}
Note that \cite[Corollary $5.4$]{bru2015} is proved when $Z$ is the plane blow up at
$n$ points on a smooth conic. For the monic log-del Pezzo pair $(Z,E)$,
Lemma $\ref{lem:deg-alg}$ can be proved similarly to
\cite[Corollary $5.4$]{bru2015}, so we omit it here.
\end{rem}

\begin{prop}\label{prop:dgfcplxmoduli-alg}
Assume that $\undl x(0)\cap(\cb P^1\times\cb P^1)=\{p\}$ and
$\undl x(0)\cap Z\neq\emptyset$. Then for a generic configuration $\undl x(0)$,
the set $\cl^\cb_*(\widetilde\zl_0,d,g,\undl x(0))$ is finite,
and only depends on $\undl x(0)$.
Moreover for an element $\bar f:\bar C\to \widetilde\zl_0$ of $\cl^\cb(\widetilde\zl_0,d,g,\undl x(0))$,
there is no irreducible component of $\bar C$ is entirely mapped into $E$.

If in addition that there is no component $\bar C^*$ of $\bar C$ such that
$\bar f(\bar C^*)=lE_i$ in $Z$ with $l\glt2$, then
the image of the irreducible component $\bar C'$ of $C_0$
whose image contains $p$ realizes either a class $l_i$ or the class $l_1+l_2$,
while any other irreducible component of $C_0$ realizes a class $l_i$.
\begin{enumerate}
  \item If $\bar f_*[\bar C']=l_i$, then the curve $C_1$ is irreducible with genus $g$,
  and $\bar f_{|C_1}$ is an element of
  $\cl^{(\delta_1,(d\cdot[E]+2k-1)\delta_1)}(Z,E,d-k[E],g,(\undl x(0)\cap Z)\cup\undl x_E)$,
  where $\undl x_E=\bar f(\bar C')\cap E$. The map $\bar f$ is the limit of a
  unique element of $\cl^\cb(\widetilde\zl_t,d,g,\undl x(t))$ as $t$ goes to $0$.
  \item If $\bar f_*[\bar C']=l_1+l_2$, then $\bar f_{|\bar C'}$ is an element of
  $\cl^{(\alpha,0)}(\cb P^1\times\cb P^1,E,l_1+l_2,0,\{p\}\cup\undl x_E)$, where
  $\undl x_E\subset\bar f(C_1)\cap E$, and $\alpha=2\delta_1$ or $\alpha=\delta_2$.
  In the former case, the curve $C_1$ either has two irreducible components $C_1'$, $C_1''$
  with $g(C_1')+g(C_1'')=g$ or is irreducible with genus $g-1$,
  and $\bar f$ is the limit of a unique element of
  $\cl^\cb(\widetilde\zl_t,d,g,\undl x(t))$ as $t$ goes to $0$;
  in the latter case, the curve $C_1$ is irreducible, and $\bar f$ is the limit of
  exactly two elements of $\cl^\cb(\widetilde\zl_t,d,g,\undl x(t))$ as $t$ goes to $0$.
\end{enumerate}
\end{prop}

\begin{proof}
Note that Proposition $\ref{prop:dgfcplxmoduli-alg}$
is the algebraic and higher genus version of \cite[Proposition $3.7$]{bp2014}.
From \cite[Proposition $2.1$]{ss2013} and Proposition $\ref{prop:hgp1p1-alg}$,
we can prove the set  $\cl^\cb_*(\widetilde\zl_0,d,g,\undl x(0))$ is finite
and only depends on $\undl x(0)$ by the standard dimension estimation.
Example $11.4$ and Lemma $14.6$ in \cite{ip2004} implies that no component
of $\bar C$ is entirely mapped into $E$.

Since $\undl x(0)\cap (\cb P^1\times\cb P^1)=\{p\}$ and $\undl x(0)\cap Z\neq\emptyset$,
the map $\bar f_{|C_1}$ is constrained by
$c_1(Z)\cdot (d-k[E])-2+g=c_1(X)\cdot d-2+g$ points in $Z$.
Hence we have the following cases:
\begin{enumerate}
  \item The curve $C_1$ is irreducible with genus $g$, and $\bar f^{-1}(E)$ consists in
  $d\cdot[E]+2k$ distinct points. The curve $C_0$ must have $d\cdot[E]+2k$ irreducible
  components with genus $0$, and the image of any of them realizes $l_i$.
  $\bar f_{|C_1}$ is also constrained by the point $\bar f(\bar C')\cap E$.
  \item $C_1$ is irreducible with genus $g-1$. $\bar f(C_1)$ intersects $E$
  transversely in $d\cdot[E]+2k$ distinct points. By adjunction formula,
  there is no genus $1$ curve represents $l_i$
  passing two fixed distinct points in $\cb P^1\times\cb P^1$.
  The curve $C_0$ must have $d\cdot[E]+2k-1$ irreducible
  components with genus $0$, one of them, say $\bar C''$, intersecting two distinct points of
  $\bar f^{-1}(E)$. $\bar f(\bar C'')$ has to realize the class $l_1+l_2$,
  and the image of any other irreducible component of $C_0$ realizes a class $l_i$.
  All the latter components are constrained by $\bar f(C_1)\cap E$, we have $\bar C'=\bar C''$.
  \item The curve $C_1$ is irreducible with genus $g$, and $\bar f^{-1}(E)$ consists in
  $d\cdot[E]+2k-1$ distinct points. The curve $C_0$ must have $d\cdot[E]+2k-1$ irreducible
  components with genus $0$. The image of one of them being tangent to $E$.
  As in the case $(2)$, we know this component must be $\bar C'$ whose image
  realizes $l_1+l_2$, and the image of any other component of $C_0$ realizes a class $l_i$.
  \item $C_1$ has two irreducible components, say $C_1'$, $C_1''$, with genus $g_1$, $g_2$,
  respectively and $g_1+g_2=g$. $\bar f(C_1)$ intersects $E$ transversely in $d\cdot[E]+2k$
  distinct points. $C_0$ must have exactly $d\cdot[E]+2k-1$ irreducible components with genus $0$,
  and the image of one of them intersecting the two components of $C_1$.
  As in the case $(2)$, we know this component must be $\bar C'$ whose image
  realizes $l_1+l_2$, and the image of any other component of $C_0$ realizes a class $l_i$.
\end{enumerate}

The statement about the number of elements of $\cl^\cb(\widetilde\zl_t,d,g,\undl x(t))$
converging to $\bar f$ as $t$ goes to $0$ follows from \cite{li2002,li2004}.
\end{proof}

Before giving the proof of Theorem $\ref{thm:gdf-alg}$,
we recall some definitions of the moduli spaces.
If the set $\undl x(t)$ of $c_1(X)\cdot d+g-1$ real holomorphic sections
$\Delta\to\widetilde\zl$ with $\undl x(0)\cap E=\emptyset$ is chosen such that
$\undl x(t)$ satisfies the additional condition $(\ref{eq:alg-number})$.
Denote by $\cl(\widetilde\zl_0,d,\undl L_\sharp,\undl{x}(0))$
the set $\{\bar f:\bar C\to \widetilde\zl_0\}$ of limits, as stable maps,
of maps in $\cl(\widetilde\zl_t,d,\undl L,\undl{x}(t))$
(see Section $\ref{subsec:hgwel-alg}$).
$\cl(\widetilde\zl_t,d,\undl L,\undl{x}(t))$
(resp. $\cl(Z,E,d,\undl L',\undl{x}(t))$ (see Section $\ref{subsec:rel wel-alg}$))
is a subset of $\rb\cl^\cb(\widetilde\zl_t,d,g,\undl{x}(t))$
(resp. $\rb\cl^{(0,d\cdot[E]\delta_1)}(Z,E,d,g-1,\undl x(t))$
which is the set consisting of real elements of $\cl^\cb(\widetilde\zl_t,d,g,\undl{x}(t))$
(resp. $\cl^{(0,d\cdot[E]\delta_1)}(Z,E,d,g-1,\undl x(t))$).

By using a real version of Proposition $\ref{prop:dgfcplxmoduli-alg}$,
we can prove Theorem $\ref{thm:gdf-alg}$ immediately.

\begin{figure}[h!]
\begin{center}
\begin{tabular}{cccc}
\includegraphics[width=3cm, angle=0]{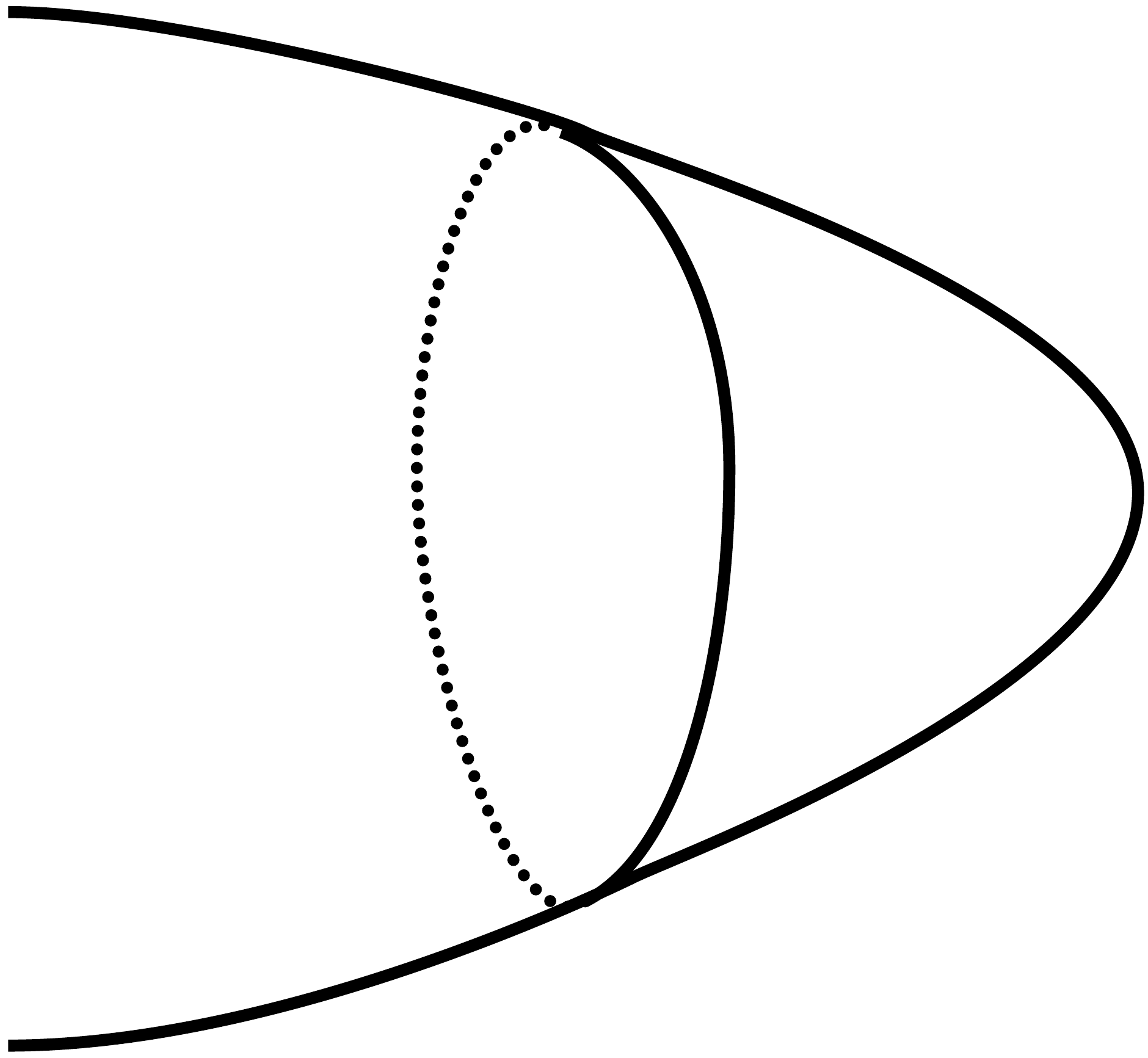}
\put(-50, 80){$L_i$}
~
\includegraphics[width=2.5cm, angle=0]{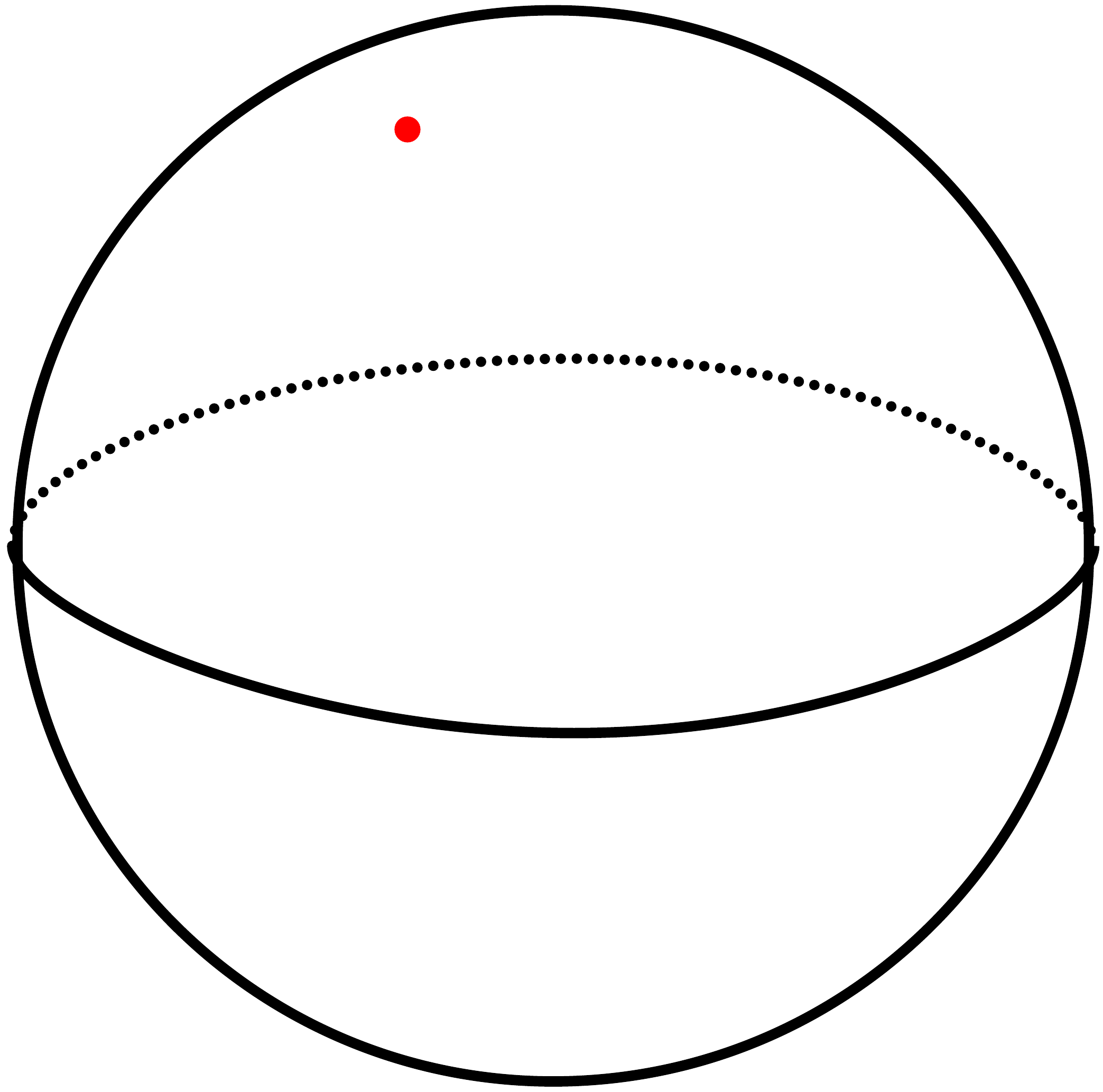}
\put(-40, 80){$S$}
& \hspace{10ex} &
\includegraphics[width=3cm, angle=0]{R_L.pdf}
\put(-50, 80){$L_i$}
~
\includegraphics[width=2.5cm, angle=0]{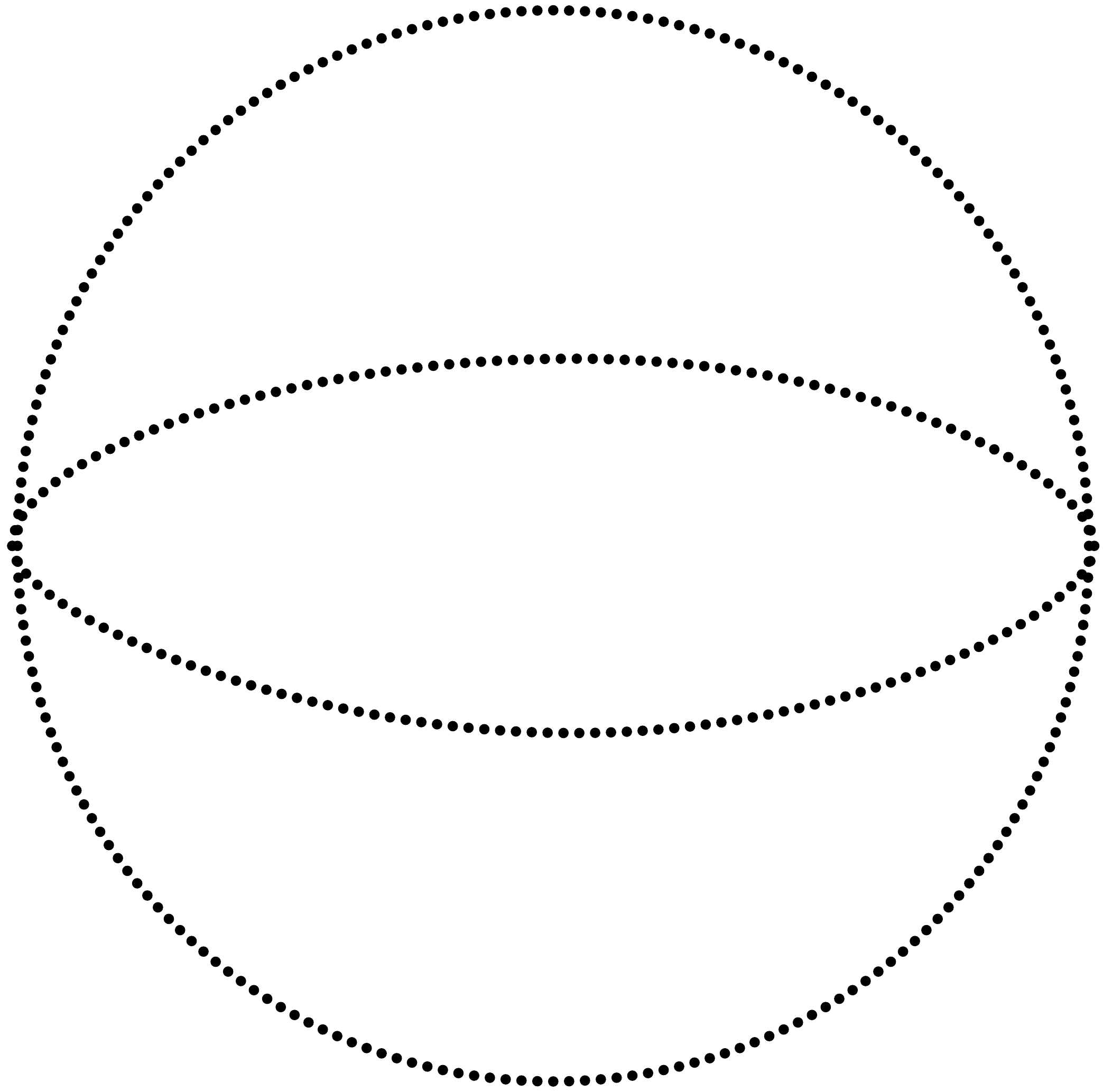}
\put(-40, 80){$S$}
\\ \\a) Part of $\rb Y$ && b) Part of $\rb X$
\end{tabular}
\end{center}
\caption{Real surgery: $Y_\rb \xrightarrow[]{S}  X_\rb$}
\label{fig:real surgery-alg}
\end{figure}

\begin{proof}[proof of Theorem $\ref{thm:gdf-alg}$]
Since the case of any real del Pezzo surface of degree $\glt3$ can be reduced to
that of degree $2$ by blowing up suitable real points, we only consider the degree $2$ case here.
According to \cite[Proposition $4.2$]{iks2013a}, we have $d\cdot[S]=0$.
Since the real configuration $\undl x$ satisfies the condition $(\ref{eq:alg-number})$
and $|S\cap\undl x|=1$ (see Figure $\ref{fig:real surgery-alg}$).
We have $\undl x(0)\cap(\cb P^1\times\cb P^1)=\{p\}$,
$\undl x(0)\cap Z\neq\emptyset$.
Choose a real element $\bar f:\bar C\to \widetilde\zl_0$ of
$\cl(\widetilde\zl_0,d,\undl L_\sharp,\undl{x}(0))\subset
\rb\cl^\cb(\widetilde\zl_0,d,g,\undl{x}(0))$.
From Proposition $\ref{prop:dgfcplxmoduli-alg}$, we know the curve $\bar f_{|C_1}$ has four possibilities.
Since $\rb E=\emptyset$, we can get that the intersection points of $E$ and $\bar f(\bar C)$
are only conjugated pairs. The existence of real curve in cases $(1)$ and $(3)$
of the proof of Proposition $\ref{prop:dgfcplxmoduli-alg}$ requires the non-emptiness
of real points of the intersection of $E$ and $\bar f(\bar C)$ (\textit{i.e.} the point belonged to $\bar f(\bar C')\cap E$).
Thus the curve $C_1$ either has two irreducible components $C_1'$, $C_1''$
with $g(C_1')+g(C_1'')=g$ or is irreducible with genus $g-1$.
Condition $(\ref{eq:alg-number})$ implies $C_1$ has non-empty real part.
If the curve $C_1$ has two irreducible components $C_1'$, $C_1''$ with
$g(C_1')+g(C_1'')=g$, and $\bar f:\bar C\to \widetilde\zl_0$ is a real curve.
The two points of intersection of $\bar f(\bar C')$
and $\bar f(C_1)$ (\textit{i.e.} the points belonged to
$(\bar f(\bar C')\cap\bar f(C_1'))\cup(\bar f(\bar C')\cap\bar f(C_1''))$) must be a pair of conjugated points.
This implies that $\bar f(C_1')$ and $\bar f(C_1'')$ are two conjugated components
which contradicts the non-emptiness of the real part.
Thus $\bar f(C_1)$ has to be a real element of
$\cl(Z,E,d-k[E],\undl L',\undl x(0)\cap Z)\subset
\rb\cl^{(0,2k\delta_1)}(Z,E,d-k[E],g-1,\undl x(0)\cap Z)$.
$\bar f(\bar C')$ realizes the class $l_1+l_2$,
and the image of any other irreducible component of $C_0$ realizes a class $l_i$.
So each curve $\bar f(C_1)$ has $k\cdot2^{k-1}$ possibilities to form a real element
$\bar f:\bar C\to \widetilde\zl_0$ of $\cl(\widetilde\zl_0,d,\undl L_\sharp,\undl{x}(0))$,
and
$$
m_{L\cup S,F}(\bar f)=m_{L,F}(\bar f_{|C_1})+k-1.
$$
We can get
\begin{equation*}
W_{Y,\undl L,F}(d,\undl r)=\sum_{k\glt1}\left((-1)^{k-1}k2^{k-1}\right)
\sum_{C_1\in\cl(Z,E,d-k[E],\undl L',\undl x(0)\cap Z)}(-1)^{m_{L,F}(C_1)}.
\end{equation*}
From the definition of relative Welschinger invariants (see equation $(\ref{eq:rel-inv-alg})$), we have
\begin{equation}\label{eq:pfthm1-alg}
W_{Y,\undl L,F}(d,\undl r)=\sum_{k\glt1}\left((-1)^{k-1}k2^{k-1}\right)
W_{Z,\undl L',F}^{E}(d-k[E],\undl r'),
\end{equation}
where $\undl r'=(r_0,\ldots,r_{g-1})$.
We use the convention
$$
\binom{-1}{0}=0.
$$
From Corollary $\ref{cor:corr-alg}$, we have
\begin{equation}\label{eq:relation-alg}
W_{Y,\undl L,F}(d,\undl r)=\sum_{k\glt1,l\glt0}(-1)^{l+k-1}k2^{k-1}
\left(\binom{k+l}{k} +\binom{k+l-1}{k}
\right)W_{X,\undl L',F}(d-(k+2l)[S],\undl r').
\end{equation}
The total coefficient of the invariant $W_{X,\undl L',F}(d-i[S],\undl r')$
appearing in the sum is equal to $1$ if $i=1$,
and is equal to $-4$ if $i=2$.
If $i\glt3$, the total coefficient of $W_{X,\undl L',F}(d-i[S],\undl r')$
appearing in the equation $(\ref{eq:relation-alg})$ is
\begin{align*}
&\sum_{k+2l=i}(-1)^{l+k-1}k2^{k-1}
\left(\binom{k+l}{k} +\binom{k+l-1}{k}
\right)\\
&=(-1)^{i-1}\sum_{k+2l=i}(-1)^{l}k2^{k-1}
\binom{k+l}{k}
-(-1)^{i-1}\sum_{k+2l_1=i-2}(-1)^{l_1}k2^{k-1}
\binom{k+l_1}{k}\\
&=(-1)^{i-1}(u_i-u_{i-2}),
\end{align*}
where
$$
u_i=\sum_{k+2l=i}(-1)^{l}k2^{k-1}\binom{k+l}{k}.
$$
Note that $u_1=1$, $u_2=4$.
Let
$$
v_i=\sum_{k+2l=i}(-1)^l2^k\binom{k+l}{k}.
$$
By Pascal's rule, we have
$$
u_{i+2}-u_{i+1}=u_{i+1}-u_{i}+v_{i+1}, i\glt1.
$$
From the proof of \cite[Theorem $1.1$]{brugalle2016},
we know $v_{i+1}=i+2$.
Hence we have
$$
u_i-u_{i-2}=i^2 ~\text{ for } i\glt3.
$$
Therefore the total coefficient of $W_{X,\undl L',F}(d-i[S],\undl r')$ appearing
in the sum is $(-1)^{i-1}i^2$ for $i\glt1$.
\end{proof}

In the following, we will give some computations of higher genus Welschinger invariants.

\begin{exa}
Let $Y_\rb$ be a real del Pezzo surface with degree $3$ whose real part
is homeomorphic to the disjoint union of a sphere $S$ and a real projective plane $\rb P^2$.
Let $X_\rb$ be the real surgery of $Y_\rb$ along $S$.
From the classification of real cubic surfaces (c.f. \cite{dk00,dk02,kollar97,sil1989}),
we may consider $X_\rb$ as $\cb P^2_{0,3}$, where $\cb P^2_{0,3}$ is the blow-up of $\cb P^2$
at $3$ pairs of complex conjugated points.
The genus $0$ Welschinger invariants of $Y_\rb$ is calculated in \cite[Section $4.1$]{brugalle2016}.
In the following, we will compute some genus $1$ Welschinger invariants of $Y_\rb$.
In $H_2(X;\zb)$, we may assume:
$$
c_1(X)=3[D]-\sum_{i=1}^{6}[E_i]~\text{and}~
[S]=2[D]-\sum_{i=1}^{6}[E_i],
$$
where $D$ is the pull back of a line in $\cb P^2$ not passing the blown up points,
and $E_i$, $i\in\{1,\ldots,6\}$ is the corresponding exceptional divisor.
The following values are taken from
\cite{abl2011,iks2004}.
\begin{center}
\begin{tabular}[c]{|c|c|c|}
\hline
  $r_0$   & $W_{X_\rb,\rb P^2,0}(4[D]-\sum_{i=1}^{6}[E_i],r_0)$&$W_{X_\rb,\rb P^2,0}(2[D],r_0)$\\
\hline
  6         & 40 & 1 \\
\hline
  4        &16& 1 \\
\hline
  2       &0& 1 \\
\hline
\end{tabular}
\end{center}
By Theorem $\ref{thm:gdf-alg}$, we have
$$
W_{Y_\rb,\rb P^2,0}(d,(r_0,1))=\sum_{k\glt1}(-1)^{(k-1)}k^2W_{X_\rb,\rb P^2,0}(d-k[S],r_0).
$$
Thus we obtain
\begin{center}
\begin{tabular}[c]{|c|c|c|c|}
\hline
  $r_0$   & 5  &3&1\\
\hline
  $W_{Y_\rb,\rb P^2\sqcup S,0}(2c_1(Y),(r_0,1))$ &36 & 12& -4 \\
\hline
\end{tabular}
\end{center}
in accordance with the genus $1$ real Welschinger invariants of del Pezzo surface of degree $\glt2$
calculated by Shustin in \cite{shustin2015}.
\end{exa}

\begin{exa}
Let $Y_\rb$ be a real del Pezzo surface with degree $2$ whose real part
is homeomorphic to $\rb P^2\sharp\rb P^2\sqcup S$, where $S$ is a real sphere.
Let $X_\rb$ be the real surgery of $Y_\rb$ along $S$.
From the classification of real del Pezzo surfaces,
we may consider $X_\rb$ as $\cb P^2_{1,3}$, where $\cb P^2_{1,3}$ is the blow-up of $\cb P^2$
at $1$ real point and $3$ pairs of complex conjugated points.
Denote by $E_1,\ldots,E_6$ the exceptional divisors
corresponding to the conjugated pairs, by $E_7$ the exceptional divisor
corresponding to the real point. Let $D$ still be the pull back of a line in
$\cb P^2$ not passing the blown up points.
From Welschinger's wall-crossing formula
$$
W_{X_\rb,\rb P^2,0}(d,r_0+2) = W_{X_\rb,\rb P^2,0}(d,r_0)+
2W_{X_{1,0},\rb P^2\sharp\rb P^2,0}(p^!d-2[E],r_0),
$$
we can get the following values.
\begin{center}
\begin{tabular}[c]{|c|c|c|}
\hline
  $r_0$   & $W_{X_\rb,\rb P^2\sharp\rb P^2,0}(4[D]-\sum_{i=1}^{6}[E_i]-2[E_7],r_0)$&
  $W_{X_\rb,\rb P^2\sharp\rb P^2,0}(2[D]-2[E_7],r_0)$\\
\hline
  3         & 12 & 0 \\
\hline
  1        &8& 0 \\
\hline
\end{tabular}
\end{center}
From Theorem $\ref{thm:gdf-alg}$, we have
\begin{center}
\begin{tabular}[c]{|c|c|c|}
\hline
  $r_0$   & 3  &1\\
\hline
  $W_{Y_\rb,\rb P^2\sharp\rb P^2\sqcup S,0}(2c_1(Y),(r_0,1))$ &12 & 8 \\
\hline
\end{tabular}
\end{center}
in accordance with the genus $1$ real Welschinger invariants of del Pezzo surface of degree $\glt2$
calculated by Shustin in \cite{shustin2015}.
\end{exa}

%

\section{Appendix: genus decreasing formula in the symplectic case}
In this Appendix, we consider the genus decreasing formula in the symplectic case by
using the definition of higher genus Welschinger invariants proposed by E. Brugall\'e \cite[Theorem $3.4$]{brugalle2016}.
We refer the readers to \cite{brugalle2016} for more details about
the real surgeries along real Lagrangian spheres and the definition of higher genus Welschinger invariants.

\begin{thm}\label{thm:gdf-symp}
Let $X_\rb$ be a compact real symplectic $4$-manifold, and $S$ be a real
Lagrangian sphere in $X_\rb$, realizing a class in $H^{-\tau_X}_2(X;\zb)$.
Denote by $Y_\rb$ the real surgery of $X_\rb$ along $S$.
Let $\undl L=(L_0,\ldots,L_g)$, $\undl L'=(L_0,\ldots,L_{g-1})$,
where $L_0,\ldots,L_{g-1}$ are $g$ distinct connected components of $\rb X_\rb$ which is disjoint from $S$,
$L_g=S^2\subset\rb Y_\rb$, and $L=\cup_{i=0}^{g-1}L_i$.
Assume $F\in H^{\tau_Y}_2(X\setminus(L\cup S);\zb/2\zb)$ is normal to $[S]$.
Suppose the class $d\in H^{-\tau_Y}_2(X;\zb)$ satisfies $c_1(X)\cdot d\neq0$, if $g=3$,
and $c_1(X)\cdot d+g-2>0$. Let $\undl r=(r_0,\ldots,r_{g-1},1)$,
$\undl r'=(r_0,\ldots,r_{g-1})$. Then
$$
W_{Y_\rb,\undl L,F}(d,\undl r)=\sum_{k\glt1}(-1)^{(k-1)}k^2W_{X_\rb,\undl L',F}(d-k[S],\undl r'),
$$
whenever the invariant $W_{Y_\rb,\undl L,F}(d,\undl r)$ is defined.
\end{thm}

\begin{rem}
The proof of Theorem $\ref{thm:gdf-symp}$ is similar to the proof of Theorem $\ref{thm:gdf-alg}$.
For the convenience of the readers, we indicate some principal points of the proof of Theorem $\ref{thm:gdf-symp}$.
\end{rem}

\subsection{Curves with tangency conditions}

In this subsection, we review some basics of the curves with tangency conditions.
The presentation here mainly follows \cite[Section 3.1]{bp2014} and \cite[Section 3.1]{brugalle2016}.

An isomorphism of two $J$-holomorphic maps
$f_1:\Sigma_1\to X$ and $f_2:\Sigma_2\to X$ is
a biholomorphism $\phi:\Sigma_1\to \Sigma_2$ such that $f_1=f_2\circ\phi$.
In the following, maps are always considered up to isomorphism.

Let $(X,\omega)$ be a compact symplectic $4$-manifold.
Assume $V=V_1\cup...\cup V_\kappa$ is a finite union of pairwise disjoint embedded
symplectic spheres with the same self-intersection number $[V_i]^2=c$, $i=1,\ldots,\kappa$.
Let $d\in H_2(X;\zb)$, $g\in\zb_{\geq0}$,
$(\alpha,\beta)=(\alpha^1,\beta^1;\alpha^2,\beta^2;...;\alpha^\kappa,\beta^\kappa)$
and $\alpha^i$, $\beta^i\in\zb_{\geq0}^\infty$ such that
\begin{equation}\label{eq:tan cond}
I\alpha^i+I\beta^i=d\cdot[V_i],~i=1,\ldots,\kappa.
\end{equation}
Choose a configuration $\undl x=\undl x^\circ\sqcup\cup_{i=1}^\kappa\undl x_{V_i}$ of points in $X$,
with $\undl x^\circ$ a configuration of
$c_1(X)\cdot d-1+g-\sum_{i=1}^\kappa d\cdot[V_i]+\sum_{i=1}^\kappa|\beta^i|$
points in $X\setminus V$,
and $\undl x_{V_i}=\{p^i_{s,t}\}_{0<t\leq\alpha^i_s,s\geq1}$ a configuration
of $|\alpha^i|$ points in $V_i$.
Given an $\omega$-tamed almost complex structure $J$ on $X$ such that
$V_i$ is $J$-holomorphic,
denote by $\cl^{(\alpha,\beta)}(X,V,d,g,\undl x,J)$ the set
of irreducible $J$-holomorphic curves $f:\Sigma_g\to X$ of genus $g$ such that

$\bullet$ $f_*[\Sigma_g]=d$;

$\bullet$ $\undl x\subset f(\Sigma_g)$;

$\bullet$ $V$ does not contain $f(\Sigma_g)$;

$\bullet$ $f(\Sigma_g)$ has order of contact $s$ with $V_i$ at each point $p^i_{s,t}$;

$\bullet$ $f(\Sigma_g)$ has order of contact $s$ with $V_i$ at exactly $\beta^i_s$
distinct points on $V_i\setminus\undl x_{V_i}$.

The set of simple maps in $\cl^{(\alpha,\beta)}(X,V,d,g,\undl x,J)$ is $0$-dimensional if the almost
complex structure $J$ is chosen to be generic.
However, $\cl^{(\alpha,\beta)}(X,V,d,g,\undl x,J)$
might contain components of positive dimension corresponding to non-simple maps.

\begin{lem}\label{lem:mspace-symp}
Suppose that $c_1(X)\cdot d\neq 0$, $2$, if $g=1$,
and $c_1(X)\cdot d\neq0$, if $g=3$.
Let $V=\cup_{i=1}^\kappa V_i$ be a finite union of pairwise disjoint embedded
symplectic spheres with the same self-intersection number $[V_i]^2=c\glt-2$, $i\in\{1,\ldots,\kappa\}$.
Assume that $|\beta^1|\glt d\cdot[V_1]-1$ and $|\beta^j|= d\cdot[V_j]$ for $j\in\{2,\ldots,\kappa\}$.
Then for a generic choice of $J$ among $\omega$-tamed almost complex structures
such that all the symplectic spheres $V_1,\ldots,V_\kappa$ are $J$-holomorphic,
the set $\cl^{(\alpha,\beta)}(X,V,d,g,\undl x,J)$ with $|\undl x^\circ|\neq0$
is finite and consisted of simple maps that are all immersions.
\end{lem}

\begin{proof}
The proof can be divided into two parts: first we prove that the set
$\cl^{(\alpha,\beta)}(X,V,d,g,\undl x,J)$
contains no element which factors through a non-trivial ramified covering;
then we prove the finiteness of $\cl^{(\alpha,\beta)}(X,V,d,g,\undl x,J)$.

{\bf Step 1}:
Note that if $|\beta^i|= d\cdot[V_i]$ for $i\in\{1,\ldots,\kappa\}$,
the first part of the proof is exactly the same as step $1$ of the proof of \cite[Lemma $3.1$]{brugalle2016}.
We omit it here.

Now we consider the case $|\beta^1|= d\cdot[V_1]-1$ and
$|\beta^j|= d\cdot[V_j]$ for $j\in\{2,\ldots,\kappa\}$.
Suppose that $\cl^{(\alpha,\beta)}(X,V,d,g,\undl x,J)$
contains a non-simple map which factors through a non-trivial ramified covering
of degree $\delta\geq2$ of a simple map $f_0:\Sigma'_{g'}\to X$ of genus $g'$.
Let $d_0$ denote the homology class $(f_0)_*[\Sigma'_{g'}]$.
By the Riemann-Hurwitz formula, we can get
\begin{equation}\label{eq:rhf-1}
g\geq\delta g'+1-\delta.
\end{equation}
Since $f_0(\Sigma'_{g'})$ passes through $c_1(X)\cdot d -2+g$ points,
we have
$$
c_1(X)\cdot d_0-1+g' \geq  c_1(X)\cdot d-2+g.
$$
Combined with equation $(\ref{eq:rhf-1})$, we have
$$
c_1(X)\cdot d_0-1+g' \leq  \frac{1}{\delta-1}.
$$
Since $|\undl x^0|>0$, we know $0<c_1(X)\cdot d_0-1+g'\leq 1$.
Therefore,
$$
c_1(X)\cdot d_0-1+g'=1.
$$
Furthermore, all inequalities above are in fact equalities.
We can get $g'\leq2$ since $c_1(X)\cdot d_0\geq0$.
If $g'=2$, we have $c_1(X)\cdot d_0=c_1(X)\cdot d=0$ and $g=3$ which contradicts the assumption.
If $g'=1$, we have $c_1(X)\cdot d=2$ and $g=1$ which is excluded by the assumption.
If $g'=0$, we obtain $g=-1$ which is a contradiction.

{\bf Step 2}:
Suppose that $\cl^{(\alpha,\beta)}(X,V,d,g,\undl x,J)$
contains infinitely many simple maps. By Gromov compactness theorem,
there exists a sequence $(f_n)_{n\glt0}$ of distinct simple maps in
$\cl^{(\alpha,\beta)}(X,V,d,g,\undl x,J)$ which converges
to some $J$-holomorphic map $\bar f:\bar C\to X$.
By genericity of $J$, the set of simple maps in
$\cl^{(\alpha,\beta)}(X,V,d,g,\undl x,J)$ is discrete.
Hence either $\bar C$ is reducible, or $\bar f$ is non-simple.
Let $\bar C_1$,...,$\bar C_m$, $\bar C_1^1$,...,$\bar C^{m_1}_1$,...
$\bar C^1_\kappa$,...,$\bar C^{m_\kappa}_\kappa$
be the irreducible components of $\bar C$, labeled in such a way that
\begin{itemize}
  \item $\bar f(\bar C_i)\nsubseteq V$ and the genus of $\bar C_i$ is $g_i$, $i=1,\ldots,m$;
  \item $\bar f(\bar C^i_j)\subset V_j$, $(\bar f)_*[\bar C^i_j]=l^i_j[V_j]$,
  $i=1,...,m_j$, $j=1,...,\kappa$, and the genus of $\bar C^i_j$ is $g^i_j$.
\end{itemize}
Define $l_j=\sum^{m_j}_{i=1}l^i_j$, $\bar g_j=\sum^{m_j}_{i=1}g^i_j$, $j=1,...,\kappa$.
Let $\bar f_*[\bar C_i]=d_i$, $i=1,...,m$.
Then $\sum_{i=1}^md_i=d-\sum_{j=1}^\kappa l_j[V_j]$.
The restriction of $\bar f$ to $\displaystyle\cup_{i=1}^m\bar C_i$ is subjected to
$c_1(X)\cdot d-1+g-\sum_{i=1}^\kappa d\cdot[V_i]+\sum_{i=1}^\kappa|\beta^i|$ points constrains,
so we have
$$
c_1(X)\cdot(d-\sum_{j=1}^\kappa l_j[V_j])-m+\sum_{i=1}^mg_i\geq
c_1(X)\cdot d-1+g-\sum_{i=1}^\kappa d\cdot[V_i]+\sum_{i=1}^\kappa|\beta^i|.
$$
Since $V_1,\ldots,V_\kappa$, are embedded symplectic spheres with the
same self-intersection $[V_j]^2=c$, from adjunction formula,
we can get $c_1(X)\cdot[V_j]=[V_j]^2+2$, $j=1,...,\kappa$.
Hence we obtain
\begin{equation}\label{eq:lem:modspace1}
2\sum_{j=1}^\kappa l_j+\sum_{j=1}^\kappa l_j[V_j]^2+m-1-\sum_{i=1}^\kappa d\cdot[V_i]+
\sum_{i=1}^\kappa|\beta^i|+g-\sum_{i=1}^mg_i\llt0.
\end{equation}
From the assumption $|\undl x^\circ|>0$, one can get $m\glt1$.
Since $\sum_{i=1}^mg_i\llt g$, we have the following cases:
\begin{enumerate}
  \item $|\beta^i|= d\cdot[V_i]$ for $i\in\{1,\ldots,\kappa\}$:
  \begin{description}
    \item[(a)] $\sum_{j=1}^\kappa l_j=0$, $m=1$, and $g_1=g$;
    \item[(b)] $[V_1]^2=-2$, $\sum_{j=1}^\kappa l_j>0$, $m=1$, and $g_1=g$.
  \end{description}
  \item $|\beta^1|= d\cdot[V_1]-1$ and $|\beta^j|= d\cdot[V_j]$ for $j\in\{2,\ldots,\kappa\}$:
  \begin{description}
    \item[(a)] $\sum_{j=1}^\kappa l_j=0$, $m=1$, and $g_1=g-1$;
    \item[(b)] $\sum_{j=1}^\kappa l_j=0$, $m=1$, and $g_1=g$;
    \item[(c)] $\sum_{j=1}^\kappa l_j=0$, $m=2$, and $g_1+g_2=g$;
    \item[(d)] $[V_1]^2=-1$, $\sum_{j=1}^\kappa l_j=1$, $m=1$, and $g_1=g$;
    \item[(e)] $[V_1]^2=-2$, $\sum_{j=1}^\kappa l_j>0$, $m=1$, and $g_1=g-1$;
    \item[(f)] $[V_1]^2=-2$, $\sum_{j=1}^\kappa l_j>0$, $m=1$, and $g_1=g$;
    \item[(g)] $[V_1]^2=-2$, $\sum_{j=1}^\kappa l_j>0$, $m=2$, and $g_1+g_2=g$.
  \end{description}
\end{enumerate}

Note that case $(1)$ is impossible which has been proved by \cite[Lemma 3.1]{brugalle2016}.
Now we end the proof of the lemma by showing case $(2)$ can be ruled out too.

\begin{description}
  \item[(2)(a)(b)] $\sum_{j=1}^\kappa l_j=0$, $m=1$, and $g_1=g-1$ or $g_1=g$:

  The curve $\bar C$ is irreducible. $\bar f$ is a non-simple map, and it factorizes
  through a non-trivial ramified covering of a simple map $f_0:C_0\to X$,
  which contradicts Step 1.
  \item[(2)(c)] $\sum_{j=1}^\kappa l_j=0$, $m=2$, and $g_1+g_2=g$:

  By genericity, the curve $\bar f(\bar C_1\cup\bar C_2)$ is fixed by the
  $c_1(X)\cdot d-2+g$ point constraints, and intersects $V$ transversely
  at non-prescribed points. This contradicts the fact that either $\alpha\neq0$ or $\beta_2\neq0$.
  \item[(2)(d)] $[V_1]^2=-1$, $\sum_{j=1}^\kappa l_j=1$, $m=1$, and $g_1=g$:

  Since $l_j\glt0$ and $\sum_{j=1}^\kappa l_j=1$, without loss of generality,
  we may assume $l_1=1$, $l_j=0$, $j=2,\ldots,\kappa$.
  By genericity of $J$, the curve $\bar f(\bar C_1)$ is fixed by the $c_1(X)\cdot d-2+g$
  point constraints in $X$, and intersects $V$ transversely in $d\cdot[V]+1$
  non-prescribed points. Any intersection point of $\bar f(\bar C_1\setminus\bar C_1^1)$
  deforms to an intersection point of the image of $f_n$ and $V$ for $n\gg1$.
  Since all intersection points of $\bar f(\bar C_1)$ and $V$ are transverse and non-prescribed,
  the component $\bar C_1^1$ contains the limit of the point corresponding
  to the extra constraint $\alpha_1$ or $\beta_2$. $f_n$ and $V$ for $n\gg1$
  must have at least $d\cdot[V]+1$ intersection points, which is a contradiction.
  \item[(2)(e)] $[V_1]^2=-2$, $\sum_{j=1}^\kappa l_j>0$, $m=1$, and $g_1=g-1$:

  Since the curve $\bar f(\bar C_1)$ is fixed by $c_1(X)\cdot d-2+g$ point constraints,
  $\bar f_{|\bar C_1}$ is simple from Step $1$ and intersects $V$ transversely.
  Any intersection point of
  $\bar f(\bar C_1\setminus(\bar C^{1}_1\cup\ldots\cup\bar C^{m_\kappa}_\kappa))$
  and $V$ deforms to an intersection point of the image of $f_n$ and $V$ for $n\gg1$.
  It follows from $(d-\sum_{j=1}^\kappa l_j[V_j])\cdot[V]=d\cdot[V]+2\sum_{j=1}^\kappa l_j$
  that at least $d\cdot[V]+\sum_{j=1}^\kappa l_j-1$ intersection points of $\bar f(\bar C_1)$
  and $V$ deform to an intersection point of the image of $f_n$ and $V$ for $n\gg1$.
  Since all intersection points of $\bar f(\bar C_1)$ and $V$ are transverse and non-prescribed,
  the components $\bar C_1^1,\ldots,\bar C^{m_\kappa}_\kappa$ contain the limit of the point corresponding
  to the extra constraint $\alpha_1$ or $\beta_2$. $f_n$ and $V$ for $n\gg1$
  must have at least $d\cdot[V]+\sum_{j=1}^\kappa l_j$ intersection points.
  But this contradicts the fact that two $J$-holomorphic curves intersect positively.
  \item[(2)(f)] $[V_1]^2=-2$, $\sum_{j=1}^\kappa l_j>0$, $m=1$, and $g_1=g$:

  Since $\bar f_{|\bar C_1}$ satisfies $c_1(X)\cdot d-2+g$ point constraints,
  $\bar f_{|\bar C_1}$ is simple from Step $1$ and it has at most one tangency point with $V$.
  The same argument used in the case $(2)(e)$ implies that $\sum_{j=1}^\kappa l_j=1$
  (Assume $l_1=1$, $l_j=0$, $j=2,\ldots,\kappa$)
  and $\bar f(\bar C_1)$ is tangent to $V$ at $\bar f(\bar C_1\cap\bar C^1_1)$.
  Hence $\bar f_{|C_1}$ is fixed by this tangency condition and the $c_1(X)\cdot d-2+g$
  other point conditions, and the component $\bar C_1^1$ contains the limit of the
  point corresponding to the extra constraint $\alpha_1$ or $\beta_2$.
  This contradicts to the positivity of intersection points of $V$ and the image of $f_n$ for $n\gg1$.
  \item[(2)(g)] $[V_1]^2=-2$, $\sum_{j=1}^\kappa l_j>0$, $m=2$, and $g_1+g_2=g$:

  By the genericity of $J$, $\bar f(\bar C_1\cup\bar C_2)$ is fixed by the $c_1(X)\cdot d-2+g$
  point constraints, and intersects $V$ transversely at non-prescribed points.
  The same argument used in the case $(2)(e)$ implies that $\sum_{j=1}^\kappa l_j=1$
  (Assume $l_1=1$, $l_j=0$, $j=2,\ldots,\kappa$).
  $\bar C_1^1$ must contain the limit of the point corresponding to the extra
  constraint $\alpha_1$ or $\beta_2$, which is a contradiction.
\end{description}
\end{proof}

\subsection{Genus decreasing formula}

Suppose that $Y_\rb$ is the real surgery of $X_\rb$ along a real Lagrangian sphere $S$,
$Y_\rb \xrightarrow[]{S}  X_\rb$ (see \cite[Section $2.2$]{brugalle2016}).
From \cite[Section $2.2$]{brugalle2016}, we know $Y_\rb$ is
the real symplectic sum of $Z_\rb$ and
$(\cb P^1\times\cb P^1,\omega_0,\tau_{\emptyset,2})$ along $E$,
where $Z_\rb$ is a real symplectic $4$-manifold containing an embedded symplectic sphere $E$ with $[E]^2=-2$,
and $\tau_{\emptyset,2}$ is a real structure on $\cb P^1\times\cb P^1$ such that
$\rb E=\emptyset$ and $\chi(\rb(\cb P^1\times\cb P^1))=2$.
Let $L=\cup_{i=0}^{g-1}L_i$ be a union of $g$ distinct connected components of $\rb X_\rb$,
$L_g=S\subset\rb Y_\rb$ be a real Lagrangian sphere disjoint from $L$,
$\undl L=(L_0,\ldots,L_{g-1},L_g)$ and $\undl L'=(L_0,\ldots,L_{g-1})$.

Let $\pi:\zl\to\Delta$ be the real symplectic sum of $Z_\rb$ and $\cb P^1\times\cb P^1$ along $E$
(see \cite[Section $2.2$]{brugalle2016}), $d\in H_2(\zl_\lambda;\zb)$.
Choose $\undl x(\lambda)$ a set of $c_1(X)\cdot d+g-1$ real
symplectic sections $\Delta\to\zl$ such that $\undl x(0)\cap E=\emptyset$.
Choose an almost complex structure $J$ on $\zl$ tamed by $\omega_\zl$, which restricts to an
almost complex structure $J_\lambda$ tamed by $\omega_\lambda$
on each fiber $\zl_\lambda$, and generic with
respect to all choices we made.

Let $X_\sharp= Z_\rb\cup_E (\cb P^1\times\cb P^1)$,
and $\undl L_\sharp$ be the degeneration of $\undl L$ as $Y_\rb$ degenerates to $X_\sharp$.
Denote by $\cl^\cb(X_\sharp,\emptyset,d,g,\undl{x}(0),J_0)$ the set
$\{\bar f:\bar C\to X_\sharp\}$ of limits, as stable maps,
of maps in $\cl^\cb(\zl_\lambda,\emptyset,d,g,\undl{x}(\lambda),J_\lambda)$ as $\lambda$ goes to $0$.
Recall that if $\lambda\neq0$, $\cl^\cb(\zl_\lambda,\emptyset,d,g,\undl{x}(\lambda),J_\lambda)$
has been defined as the set of irreducible $J$-holomorphic curves $f:\Sigma_g\to\zl_\lambda$ of genus $g$,
with $f_*[\Sigma_g]=d$, and passing through all points in $\undl x(\lambda)$
(see \cite[Section $3.1$]{brugalle2016}).
From \cite[Section $3$]{ip2004}, we know $\bar C$ is a connected nodal curve of
arithmetic genus $g$ such that:

$\bullet$ $\undl x(0)\subset\bar f(\bar C)$;

$\bullet$ any point $p\in\bar f^{-1}(E)$ is a node of $\bar C$ which is the intersection of two
irreducible components $\bar C'$ and $\bar C''$ of $\bar C$, with $\bar f(\bar C')\subset Z$
and $\bar f(\bar C'')\subset\cb P^1\times\cb P^1$;

$\bullet$ if in addition neither $\bar f(\bar C')$ nor $\bar f(\bar C'')$ is entirely mapped into $E$,
then the multiplicities of intersection of both $\bar f(\bar C')$ and $\bar f(\bar C'')$ with $E$ are equal.

Given an element
$\bar f:\bar C\to X_\sharp$ of $\cl^\cb(X_\sharp,\emptyset,d,g,\undl x(0),J_0)$,
denote by $C_1$ (resp. $C_0$)
the union of the irreducible components of $\bar C$ mapped into $Z$ (resp. $\cb P^1\times\cb P^1$).

\begin{lem}\cite[Lemma $3.12$]{brugalle2016}\label{lem:abv}
Given an element $\bar f:\bar C\to X_\sharp$ of $\cl^\cb(X_\sharp,\emptyset,d,g,\undl x(0),J_0)$,
there exists $k\in\zb_{\glt0}$ such that
$$
\bar f_*[C_1]=d-k[E]~\textrm{ and }~\bar f_*[C_0]=kl_1+(d\cdot[E]+k)l_2.
$$
Moreover $c_1(Z)\cdot\bar f_*[C_1]=c_1(\zl_\lambda)\cdot d$.
\end{lem}

\begin{prop}\label{prop:dgfcplxmoduli-symp}
Assume that $\undl x(0)\cap(\cb P^1\times\cb P^1)=\{p\}$ and
$\undl x(0)\cap Z\neq\emptyset$. Then for a generic $J_0$,
the set $\cl^\cb(X_\sharp,\emptyset,d,g,\undl x(0),J_0)$ is finite,
and only depends on $\undl x(0)$ and $J_0$.
Given an element $\bar f:\bar C\to  X_\sharp$ of $\cl^\cb(X_\sharp,\emptyset,d,g,\undl x(0),J_0)$,
the restriction of $\bar f$ to any component of $\bar C$ is a simple map,
and no irreducible component of $\bar C$ is entirely mapped into $E$.
Moreover the image of the irreducible component $\bar C'$ of $C_0$
whose image contains $p$ realizes either a class $l_i$ or the class $l_1+l_2$,
while any other irreducible component of $C_0$ realizes a class $l_i$.
\begin{enumerate}
  \item If $\bar f_*[\bar C']=l_i$, then the curve $C_1$ is irreducible with genus $g$,
  and $\bar f_{|C_1}$ is an element of
  $\cl^{(\delta_1,(d\cdot[E]+2k-1)\delta_1)}(Z,E,d-k[E],g,(\undl x(0)\cap Z)\cup\undl x_E,J_0)$,
  where $\undl x_E=\bar f(\bar C')\cap E$. The map $\bar f$ is the limit of a
  unique element of $\cl^\cb(\zl_\lambda,\emptyset,d,g,\undl x(\lambda),J_\lambda)$
  as $\lambda$ goes to $0$.
  \item If $\bar f_*[\bar C']=l_1+l_2$, then $\bar f_{|\bar C'}$ is an element of
  $\cl^{(\alpha,0)}(\cb P^1\times\cb P^1,E,l_1+l_2,0,\{p\}\cup\undl x_E,J_0)$, where
  $\undl x_E\subset\bar f(C_1)\cap E$, and $\alpha=2\delta_1$ or $\alpha=\delta_2$.
  In the former case, the curve $C_1$ either has two irreducible components $C_1'$, $C_1''$
  with $g(C_1')+g(C_1'')=g$ or is irreducible with genus $g-1$,
  and $\bar f$ is the limit of a unique element of
  $\cl^\cb(\zl_\lambda,\emptyset,d,g,\undl x(\lambda),J_\lambda)$ as $\lambda$ goes to $0$;
  in the latter case, the curve $C_1$ is irreducible, and $\bar f$ is the limit of
  exactly two elements of $\cl^\cb(\zl_\lambda,\emptyset,d,g,\undl x(\lambda),J_\lambda)$
  as $\lambda$ goes to $0$.
\end{enumerate}
\end{prop}

\begin{rem}
Note that Proposition $\ref{prop:dgfcplxmoduli-symp}$
is a higher genus version of \cite[Proposition $3.7$]{bp2014}.
If we replace \cite[Proposition $3.3$]{bp2014} by Lemma $\ref{lem:mspace-symp}$,
we can prove Proposition $\ref{prop:dgfcplxmoduli-symp}$ similarly to \cite[Proposition $3.7$]{bp2014}.
Therefore, we omit it here.
The readers may refer to \cite[Proposition $3.7$]{bp2014}
and Proposition $\ref{prop:dgfcplxmoduli-alg}$ for more details.
\end{rem}

If the set $\undl x(\lambda)$ of $c_1(X)\cdot d+g-1$ real symplectic sections
$\Delta\to\zl$ with $\undl x(0)\cap E=\emptyset$ is chosen such that
$\undl x(\lambda)$ satisfies the additional condition:
\begin{equation}\label{eq:x-L-symp}
\text{either  } ~~g = 0~~ \text{  or   }~~ |L_i\cap\undl x|=l_{L_i,d}^2+1 \mod 2~~ \forall i\in\{0,...,g\},
\end{equation}
where $l_{L_i,d}=[f(\rb\Sigma_g\cap L_i)]\in H_1(L_i,\zb/2\zb)$ (see \cite[Lemma $3.2$]{brugalle2016}).
Denote by $\cl(X_\sharp,\undl L_\sharp,d,\undl x(0),J_0)$
the set $\{\bar f:\bar C\to X_\sharp\}$ of limits, as stable maps,
of maps in $\cl(\zl_\lambda, \undl L,d,\undl x(\lambda),J_\lambda)$,
where $\cl(\zl_\lambda, \undl L,d,\undl x(\lambda),J_\lambda)$
is the set of real $J_\lambda$-holomorphic curves $f:\Sigma_g\to \zl_\lambda$ of genus $g$
in $\zl_\lambda$, with $f_*[\Sigma_g]=d$, passing through $\undl x(\lambda)$,
and such that $f(\rb\Sigma_g)\subset L\cup S$
(see \cite[Section $3.2$]{brugalle2016} for more details).
Denote by $\cl(Z,E,d,\undl L',\undl x(\lambda)\cap Z,J_\lambda)$ the set of real $J_\lambda$-holomorphic curves
$f:\Sigma_{g-1} \to Z$ of genus $g-1$, with $f_*[\Sigma_{g-1}]=d$, whose image is not contained in $E$,
passing through $\undl{x}(\lambda)$, and such that $f(\rb\Sigma_{g-1})\subset L$
(see \cite[Section $3.3$]{brugalle2016} for more details).
$\cl(\zl_\lambda, \undl L,d,\undl x(\lambda),J_\lambda)$
(resp. $\cl(Z,E,d,\undl L',\undl x(\lambda)\cap Z,J_\lambda)$)
is a subset of $\rb\cl^\cb(\zl_\lambda,\emptyset,d,g,\undl x(\lambda),J_\lambda)$
(resp. $\rb\cl^{(0,d\cdot[E]\delta_1)}(Z,E,d,g-1,\undl x(\lambda)\cap Z,J_\lambda)$)
which is the set consisting of real elements of
$\cl^\cb(\zl_\lambda,\emptyset,d,g,\undl x(\lambda),J_\lambda)$
(resp. $\cl^{(0,d\cdot[E]\delta_1)}(Z,E,d,g-1,\undl x(\lambda)\cap Z,J_\lambda)$).
By using a real version of Proposition $\ref{prop:dgfcplxmoduli-symp}$
and the symplectic sum formula \cite{egh2000,ip2004,lr2001,li2002},
we can prove Theorem $\ref{thm:gdf-symp}$ in a similar way to the proof of Theorem $\ref{thm:gdf-alg}$.


\end{document}